\documentclass[11pt]{amsart}

\usepackage[T1]{fontenc}
\usepackage[utf8]{inputenc}
\usepackage{lmodern}
\usepackage{amsmath,amssymb,amsthm,mathtools}
\usepackage{enumitem}
\usepackage{array}
\usepackage{microtype}
\usepackage[hidelinks]{hyperref}

\newtheorem{theorem}{Theorem}[section]
\newtheorem{lemma}[theorem]{Lemma}
\newtheorem{proposition}[theorem]{Proposition}
\newtheorem{corollary}[theorem]{Corollary}
\theoremstyle{definition}
\newtheorem{definition}[theorem]{Definition}
\theoremstyle{remark}

\newcommand{\Pow}{\mathcal P}
\newcommand{\Cpx}{\mathfrak C}
\newcommand{\Dwn}{\operatorname{Down}}

\newcommand{\Jp}{\operatorname{Jp}}

\newcommand{\eps}{\varepsilon}

\newcommand{\T}{\mathbf T}
\newcommand{\Con}{\operatorname{Con}}
\newcommand{\Aut}{\operatorname{Aut}}

\title[Typed congruences in ternary Gamma-semirings]
{Typed Congruences and Quotient-Critical Irreducibility in Complete Ternary
$\Gamma$-Semirings}

\author{Chandrasekhar Gokavarapu}
\address{Department of Mathematics, Government College (Autonomous),
Rajahmundry, Andhra Pradesh 533105, India; and Department of Mathematics,
Acharya Nagarjuna University, Guntur, Andhra Pradesh 522510, India}
\email{chandrasekhargokavarapu@gmail.com}
\email{chandrasekhargokavarapu@gcrjy.ac.in}
\thanks{ORCID: 0009-0006-5306-371X}

\author{Madhusudhana Rao Dasari}
\address{Department of Mathematics, Acharya Nagarjuna University, Guntur,
Andhra Pradesh 522510, India; and Department of Mathematics, Government
College for Women (Autonomous), Pattabhipuram, Guntur, Andhra Pradesh 522006,
India}
\email{dmrmaths@gmail.com}
\thanks{ORCID: 0000-0003-0897-868X}

\date{August 24, 2026}

\hypersetup{
  pdftitle={Typed Congruences and Quotient-Critical Irreducibility in Complete Ternary Gamma-Semirings},
  pdfauthor={Chandrasekhar Gokavarapu and Madhusudhana Rao Dasari},
  pdfkeywords={ternary Gamma-semiring, polyadic semigroup, typed congruence, quotient, binary reducibility}
}

\begin{document}
\sloppy
\begin{abstract}
We study congruences and binary reducibility in completely
additive ternary $\Gamma$-semirings through their two-sorted atomic cores.
For an odd $m\geq3$, an abelian group $G$ of exponent dividing $m-1$, and
$u,v\in G$, we construct a two-branch symmetric $m$-ary band
$F_m(G;u,v)$.  We prove
\[
\Con(F_m(G;u,v))\cong\operatorname{Sub}(G)\times B_2
\]
and classify every carrier--index congruence pair: $(\theta,\phi)$ is typed
exactly when $\theta$ is an ordinary congruence and $\phi$ refines $\theta$.
We also give a quotient-by-quotient reducibility criterion.  In the finite
case, $F_m(G;u,v)$ is quotient-critically irreducible exactly when $G$ is a
cyclic $p$-group and $u-v$ has order $p$.  The specialization
$G=\mathbb Z_2$ yields a four-point family $H_m$ extending the irreducible
ternary example of Devillet and Mathonet.  Its congruence lattice is $B_3$,
its alternating core has exactly $36$ typed congruence pairs, and its
automorphism group is $C_2$.  Its seven proper quotients are reducible and
form five isomorphism types; in arity five their exact reduction counts are
determined.  The powerset lift of $H_5$ is a $16$-element atom-total complete
atomic Boolean ternary $\Gamma$-semiring with no atomic binary collapse,
whereas each proper diagonal strong atom-saturated quotient has one.  Its
parity congruence is recovered on atoms by one explicitly specified
depth-one polynomial inequation.
\end{abstract}


\maketitle

\section{Introduction}

A recurring problem for higher-arity algebra is whether an associative
$m$-ary operation is induced by an associative binary operation on the same
set.  Quotients make this problem subtler: irreducibility of an operation need
not persist after states are identified.  In an indexed ternary product
\[
S\times\Gamma\times S\times\Gamma\times S\longrightarrow S
\]
there is a second issue.  Carrier and index states have different types, so a
quotient is governed by a pair of equivalence relations rather than by one
ordinary congruence.  The interaction between these two relations is not
visible in a one-sorted treatment.

Completely additive products on prime-generated complete lattices admit a
downset relational representation.  In those coordinates, the questions
studied here become finite and explicit: which typed congruences occur, which
quotient cores arise, and at exactly which quotients binary reducibility
appears.

\subsection{Context and contribution}

Polyadic algebra begins with D{\"o}rnte and Post~\cite{Dornte1929,Post1940}.
Reducibility is known for quasitrivial $n$-ary semigroups
\cite{CouceiroDevillet2019}, for idempotent nondecreasing operations on chains
\cite{KissSomlai2019}, and for broader idempotent classes
\cite{CouceiroDevilletMarichalMathonet2022}.  Symmetric
$n$-ary bands admit a strong-semilattice decomposition and an exact
reducibility criterion~\cite{DevilletMathonet2021}.  Devillet and Mathonet's
four-point ternary example is irreducible, but its congruence lattice, its
two-sorted congruence pairs, and the reducibility of all its quotients are not
part of that criterion.  Finite classification of semigroup-like structures
is also an active theme in this journal; a recent example is the classification
and computer-assisted enumeration of small doppelsemigroups
\cite{Gavrylkiv2026}.

Ternary rings, $\Gamma$-rings and ternary semirings were developed in
\cite{Nobusawa1964,Barnes1966,Lister1971,DuttaKar2003,DuttaKar2004}; the
convention used below follows~\cite{ChandrasekharRaoPrasad2025}.  Heterogeneous
algebra provides the general many-sorted language~\cite{BirkhoffLipson1970}.
For the present problem, carrier and index equivalences must be classified
simultaneously, and the resulting quotients must be tested for binary
reduction.

The supporting representation lies in the complex-algebra tradition
\cite{JonssonTarski1951,Goldblatt1989} and uses the downset description of
prime-generated complete lattices~\cite{DaveyPriestley2002,GierzEtAl2003}.
It converts atomic quotient questions for complete ternary
$\Gamma$-semirings into finite congruence questions for alternating cores.

The results are organized in three layers.
\begin{enumerate}[label=\textnormal{(\roman*)},leftmargin=*]
\item For a two-branch extension $F_m(G;u,v)$ we determine the entire
congruence lattice as $\operatorname{Sub}(G)\times B_2$ and classify all typed
congruence pairs.  We then characterize reducibility of every quotient and,
for finite $G$, characterize quotient-critical irreducibility by a cyclic
$p$-group condition.
\item Specializing to $G=\mathbb Z_2$ gives a four-point family $H_m$ with
congruence lattice $B_3$, automorphism group $C_2$, and exactly $36$ typed
congruence pairs.  Its seven proper quotients are reducible and form five
isomorphism types; their exact binary-reduction counts are obtained for
$m=5$.
\item The powerset lift of $H_5$ gives a complete atomic ternary
$\Gamma$-semiring in which binary collapse appears precisely after passing to
a proper diagonal strong atom-saturated quotient.  A specified polynomial
inequation recovers the parity congruence on atoms.
\end{enumerate}

The representation and first-isomorphism propositions supply the passage
between finite cores and complete algebras.  The two-branch classification and
the quotient-criticality criterion provide the structural results used in the
four-point and complete-algebra specializations.

\section{Two-sorted relational coordinatization}\label{sec:coord}

\begin{definition}\label{def:tgs}
A \emph{ternary $\Gamma$-semiring} consists of additive commutative semigroups
$(S,+)$ and $(\Gamma,+)$ and a map
\[
(x,\alpha,y,\beta,z)\longmapsto x\alpha y\beta z
\quad(S\times\Gamma\times S\times\Gamma\times S\to S)
\]
that is distributive in the three carrier coordinates and satisfies
\begin{equation}
(x\alpha y\beta z)\gamma u\delta v
=x\alpha(y\beta z\gamma u)\delta v
=x\alpha y\beta(z\gamma u\delta v).
\label{eq:tgs-assoc}
\end{equation}
It is \emph{completely additive} if both reducts are complete lattices with
addition equal to join and the product preserves arbitrary joins in all five
coordinates.
\end{definition}

An element $p$ of a complete lattice $L$ is \emph{completely join-prime} if
$p\ne\bot$ and $p\leq\bigvee A$ implies $p\leq a$ for some $a\in A$.
Write $\Jp(L)$ for the poset of these elements.  The lattice is
\emph{prime-generated} if every element is the join of the completely
join-prime elements below it.  For a poset $P$, let $\Dwn(P)$ denote its
complete lattice of downsets.

\begin{lemma}[Downset coordinatization]\label{lem:downset}
If $L$ is a prime-generated complete lattice, then
\[
\eta_L(x)=\{p\in\Jp(L):p\leq x\}
\]
is a complete-lattice isomorphism $L\cong\Dwn(\Jp(L))$.  The completely
join-prime elements of $\Dwn(P)$ are the principal downsets.
\end{lemma}

\begin{proof}
Prime generation makes $\eta_L$ injective, and complete join-primality gives
$\eta_L(\bigvee A)=\bigcup_{a\in A}\eta_L(a)$.  If $U$ is a downset of
$\Jp(L)$, then $\eta_L(\bigvee U)=U$, proving surjectivity.  Finally,
$U=\bigcup_{p\in U}\mathord\downarrow p$ shows that a completely join-prime
downset must be principal.  This standard lattice fact is recorded to fix the
coordinates used below~\cite{DaveyPriestley2002,GierzEtAl2003}.
\end{proof}

Let $X,\Lambda$ be posets and
\[
R\subseteq X\times\Lambda\times X\times\Lambda\times X\times X.
\]
Write $R(x,\lambda,y,\mu,z;r)$, with the last entry the output.  We require
$R$ to be isotone in its five inputs and antitone in its output.  For nine
alternating inputs define
\begin{align}
L_R(r)&\Longleftrightarrow\exists u\,[R(x_1,\lambda_1,x_2,\lambda_2,x_3;u)
 \land R(u,\lambda_3,x_4,\lambda_4,x_5;r)],\notag\\
M_R(r)&\Longleftrightarrow\exists u\,[R(x_2,\lambda_2,x_3,\lambda_3,x_4;u)
 \land R(x_1,\lambda_1,u,\lambda_4,x_5;r)],\label{eq:relassoc}\\
N_R(r)&\Longleftrightarrow\exists u\,[R(x_3,\lambda_3,x_4,\lambda_4,x_5;u)
 \land R(x_1,\lambda_1,x_2,\lambda_2,u;r)].\notag
\end{align}
An \emph{ordered alternating quinary frame} is such a relation for which
$L_R(r),M_R(r),N_R(r)$ are equivalent for all inputs and outputs.

For downsets $A,C,E\subseteq X$ and $B,D\subseteq\Lambda$, set
\begin{align}
R^\#[A,B,C,D,E]=\{r:\ &R(a,\lambda,c,\mu,e;r)
\text{ for some }a\in A,c\in C,e\in E,\notag\\[-2mm]
&\lambda\in B,\mu\in D\}.\label{eq:complex}
\end{align}
Its \emph{downset complex} is
\[
\Cpx(R)=(\Dwn(X),\cup,\varnothing;
\Dwn(\Lambda),\cup,\varnothing;R^\#).
\]

\begin{proposition}[Relational coordinatization]\label{prop:coord}
An order-compatible relation $R$ satisfies~\eqref{eq:relassoc} if and only if
$\Cpx(R)$ is a completely additive ternary $\Gamma$-semiring.  Conversely,
if $\T=(S,\Gamma)$ is a completely additive ternary $\Gamma$-semiring whose
two reducts are prime-generated, then
\begin{equation}
R_{\T}(a,\lambda,b,\mu,c;r)
\quad\Longleftrightarrow\quad r\leq a\lambda b\mu c
\label{eq:canonicalR}
\end{equation}
on $X=\Jp(S)$ and $\Lambda=\Jp(\Gamma)$ defines an ordered alternating
quinary frame, and
\[
(\eta_S,\eta_\Gamma):\T\longrightarrow\Cpx(R_{\T})
\]
is an isomorphism.
\end{proposition}

\begin{proof}
The operation~\eqref{eq:complex} preserves arbitrary unions.  On principal
downsets, membership of $r$ in the three nested products is respectively
$L_R(r),M_R(r),N_R(r)$.  Since every downset is a union of principal
downsets, complete additivity extends the equivalence to arbitrary inputs.

For~\eqref{eq:canonicalR}, expand every input and every intermediate product
as a join of completely join-prime elements.  A prime $r$ lies below a nested
product exactly when an intermediate prime witnesses the corresponding line
of~\eqref{eq:relassoc}.  Equation~\eqref{eq:tgs-assoc} therefore gives the
relational identities.  A second prime expansion gives
\[
\eta_S(x\alpha y\beta z)
=R_{\T}^{\#}[\eta_S(x),\eta_\Gamma(\alpha),\eta_S(y),
\eta_\Gamma(\beta),\eta_S(z)].
\]
Lemma~\ref{lem:downset} completes the proof.
\end{proof}

When both orders are discrete, downsets are arbitrary subsets.  A discrete
frame is \emph{functional} if every input tuple has exactly one output,
written $Q(x,\lambda,y,\mu,z)$.  Then~\eqref{eq:relassoc} is precisely the
three typed associativity identities for $Q$.

We call a completely additive complete atomic algebra \emph{atom-total} if
every product of three carrier atoms and two index atoms is a carrier atom.

\begin{corollary}\label{cor:atomiccoord}
Completely additive ternary $\Gamma$-semirings with complete atomic Boolean
reducts are exactly the powerset complexes of discrete relational frames.
The canonical frame is functional exactly when every product of three carrier
atoms and two index atoms is a carrier atom.
\end{corollary}

\section{Typed congruences and atomic quotients}\label{sec:typedquot}

From now on frames are discrete.  A pair $(\theta,\phi)$ of equivalence
relations on $X$ and $\Lambda$ is a \emph{strong frame congruence} if
\[
\{[r]_\theta:R(x,\lambda,y,\mu,z;r)\}
\]
depends only on $[x]_\theta,[y]_\theta,[z]_\theta$ and
$[\lambda]_\phi,[\mu]_\phi$.  The quotient relation is
\begin{align}
\overline R([x],[\lambda],[y],[\mu],[z];[r])
\Longleftrightarrow [r]\in
\{[s]:R(x,\lambda,y,\mu,z;s)\}.
\label{eq:quotR}
\end{align}
For a functional core $Q$, this condition reduces to
\begin{equation}
x_i\mathrel\theta x_i',\quad
\lambda_j\mathrel\phi\lambda_j'
\Longrightarrow
Q(x_1,\lambda_1,x_2,\lambda_2,x_3)
\mathrel\theta
Q(x_1',\lambda_1',x_2',\lambda_2',x_3').
\label{eq:typedcong}
\end{equation}
Thus it is the ordinary congruence condition for a two-sorted algebra
\cite{BirkhoffLipson1970,BurrisSankappanavar1981}.

For $A,A'\subseteq X$, put
\[
A\equiv_\theta A'\quad\Longleftrightarrow\quad
\{[a]_\theta:a\in A\}=\{[a']_\theta:a'\in A'\},
\]
and define $\equiv_\phi$ analogously.

If $R$ and $T$ are discrete frames, a pair of maps $(f,g)$ on their carrier
and index sorts is a \emph{strong morphism} when, for every input tuple,
\[
\{f(r):R(x,\lambda,y,\mu,z;r)\}
=\{s:T(fx,g\lambda,fy,g\mu,fz;s)\}.
\]
For a functional frame this is the usual two-sorted homomorphism condition.
We call a congruence pair of a powerset complex \emph{atom-saturated} when it
has the form $(\equiv_\theta,\equiv_\phi)$ for equivalences on the two sets of
atoms.

\begin{proposition}[Typed quotient and first isomorphism]\label{prop:firstiso}
If $(\theta,\phi)$ is a strong frame congruence, then~\eqref{eq:quotR} is well
defined, the quotient is a discrete alternating quinary frame, and the
canonical pair of maps is strong.  Moreover,
\[
\Cpx(R)/(\equiv_\theta,\equiv_\phi)
\cong\Cpx(R/(\theta,\phi)).
\]
If $(f,g):R\to T$ is a surjective strong morphism, then
$(\ker f,\ker g)$ is a strong frame congruence and
$R/(\ker f,\ker g)\cong T$.
\end{proposition}

\begin{proof}
Strong-congruence invariance gives well-definedness.  In a nested quotient
product, representatives of the intermediate class can be replaced without
changing the possible output classes; consequently its three relational
composites are the images of $L_R,M_R,N_R$.  Quotient associativity follows.
On powersets the maps $A\mapsto\{[a]:a\in A\}$ preserve unions and the
complex product and have kernels $\equiv_\theta,\equiv_\phi$.  The final
statement follows because strongness identifies output classes with target
outputs and surjectivity identifies the quotient sorts with the target sorts.
\end{proof}

The proposition is the bridge used later: classifying strong congruence pairs
of a finite core classifies the corresponding atom-saturated complete
quotients of its powerset algebra.

\section{Binary reductions and quotient-criticality}\label{sec:binary}

A diagonal functional core $(X,X,Q)$ is \emph{binary reducible} if an
associative operation $g:X^2\to X$ satisfies
\begin{equation}
Q(x_1,x_2,x_3,x_4,x_5)
=g(g(g(g(x_1,x_2),x_3),x_4),x_5).
\label{eq:reduction}
\end{equation}
An \emph{atomic binary collapse} of $\Cpx(Q)$ is an associative,
arbitrary-union-preserving operation $\star$ on $\Pow(X)$ that sends pairs of
singletons to singletons and whose fivefold product is the complex operation.

\begin{lemma}\label{lem:collapse}
Binary reductions of $Q$ correspond bijectively to atomic binary collapses of
$\Cpx(Q)$ by
\[
A\star B=\{g(a,b):a\in A,\ b\in B\}.
\]
\end{lemma}

\begin{proof}
Setwise extension preserves associativity and unions and gives the required
fivefold product.  Conversely, restriction of a collapse to singleton pairs
defines $g$; associativity and~\eqref{eq:reduction} follow on singletons.
Complete union preservation recovers the collapse from this restriction.
\end{proof}

\begin{proposition}[Descent criterion]\label{prop:descent}
Let $g$ reduce $Q$ and let $(\theta,\theta)$ be a strong diagonal congruence.
The quotient $Q/\theta$ is reduced by
\[
\overline g([x],[y])=[g(x,y)]
\]
if and only if $\theta$ is a semigroup congruence of $(X,g)$.  Equivalently,
the associated atomic collapse descends to $\Pow(X/\theta)$.
\end{proposition}

\begin{proof}
The formula for $\overline g$ is well defined exactly when $\theta$ is a
semigroup congruence.  Under this condition, passing~\eqref{eq:reduction} to
classes proves that $\overline g$ reduces $Q/\theta$.  The powerset assertion
then follows by setwise extension, and its restriction to singletons gives
the converse.
\end{proof}

An irreducible $m$-ary semigroup is called \emph{quotient-critically
irreducible} if each quotient by a nonidentity congruence is binary reducible.
It combines source irreducibility with the appearance of binary reductions
after every nontrivial identification; those quotient reductions need not
descend from a binary operation on the source.

\section{Two-branch extensions of abelian groups}\label{sec:twobranch}

Fix an odd integer $m\geq3$.  Let $(G,+,0)$ be an abelian group whose
exponent divides $m-1$, and choose $u,v\in G$.  On the disjoint union
\[
X(G)=G\mathbin{\dot\cup}\{a,b\}
\]
define $F_m(G;u,v):X(G)^m\to X(G)$ by
\begin{equation}
F_m(x_1,\ldots,x_m)=
\begin{cases}
a,&x_1=\cdots=x_m=a,\\
b,&x_1=\cdots=x_m=b,\\
\tau(x_1)+\cdots+\tau(x_m),&\text{otherwise},
\end{cases}
\label{eq:generalF}
\end{equation}
where $\tau(a)=u$, $\tau(b)=v$, and $\tau(g)=g$ for $g\in G$.
This is the strong $m$-ary semilattice over two singleton components and the
bottom group component $G$, with connecting maps $a\mapsto u$ and
$b\mapsto v$.

For a subgroup $N\leq G$ and $\varepsilon,\delta\in\{0,1\}$, let
$\theta_N^{\varepsilon,\delta}$ be the following equivalence on $X(G)$.  Its
restriction to $G$ is congruence modulo $N$.  If $\varepsilon=1$, adjoin $a$
to the coset $u+N$; otherwise $a$ is a singleton.  If $\delta=1$, adjoin $b$
to $v+N$; otherwise $b$ is a singleton.  In particular, when both flags are
one and $u-v\in N$, the elements $a$ and $b$ belong to the same class.
Write $B_k$ for the Boolean lattice on $k$ atoms.

\begin{theorem}[Congruences of two-branch extensions]
\label{thm:generalcongruences}
The operation $F_m(G;u,v)$ is a symmetric idempotent $m$-ary semigroup, and
\begin{equation}
\Con(F_m(G;u,v))
=\{\theta_N^{\varepsilon,\delta}:N\leq G,
\ \varepsilon,\delta\in\{0,1\}\}.
\label{eq:generalconlist}
\end{equation}
Moreover,
\[
(N,\varepsilon,\delta)\longmapsto
\theta_N^{\varepsilon,\delta}
\]
is a lattice isomorphism
\[
\operatorname{Sub}(G)\times B_2
\cong\Con(F_m(G;u,v)).
\]
Regard $F_m(G;u,v)$ as an alternating two-sorted operation on two copies of
$X(G)$.  A carrier--index pair $(\theta,\phi)$ is a typed congruence if and
only if
\begin{equation}
\theta\in\Con(F_m(G;u,v))
\qquad\text{and}\qquad \phi\subseteq\theta.
\label{eq:generaltyped}
\end{equation}
\end{theorem}

\begin{proof}
The $m$-fold group sum on $G$ is idempotent because $(m-1)g=0$ for every
$g\in G$.  The maps from the singleton components are $m$-ary
homomorphisms for the same reason.  Thus~\eqref{eq:generalF} is the strong
$m$-ary semilattice determined by these data, so it is associative; symmetry
and idempotence are also immediate from the formula
\cite[Proposition~4.11]{DevilletMathonet2021}.

Let $\rho$ be a congruence and restrict it to $G$.  Since
\[
F_m(g,h,0,\ldots,0)=g+h,
\]
this restriction is a group congruence, hence congruence modulo a unique
subgroup $N\leq G$.  If $a\mathrel\rho g$ for some $g\in G$, comparison in
the context $(\,\cdot\,,0,\ldots,0)$ gives $u\mathrel\rho g$; therefore
$g\in u+N$.  Conversely, transitivity shows that attachment of $a$ to one
member of $u+N$ attaches it to that entire coset.  The same argument applies
to $b$ and $v+N$.  Finally, if $a\mathrel\rho b$, the zero context gives
$u-v\in N$, while comparison of the all-$a$ tuple with a tuple containing one
$b$ gives $a\mathrel\rho v$.  Thus both exceptional elements are attached to
their indicated cosets.  It follows that $\rho$ has exactly the form
$\theta_N^{\varepsilon,\delta}$.

Each listed equivalence is a congruence.  For
$(\varepsilon,\delta)=(0,0)$ it is the kernel of the natural homomorphism
\[
F_m(G;u,v)\longrightarrow F_m(G/N;u+N,v+N)
\]
that fixes $a,b$.  If only $a$ is attached, map $a$ to $u+N$, map $G$ to
$G/N$, and retain $b$ as the single exceptional branch; the resulting map is
a homomorphism by~\eqref{eq:generalF} and has kernel
$\theta_N^{1,0}$.  The case $\theta_N^{0,1}$ is symmetric.  If both are
attached, the map
\[
g\mapsto g+N,\qquad a\mapsto u+N,\qquad b\mapsto v+N
\]
onto the $m$-fold group operation on $G/N$ has kernel
$\theta_N^{1,1}$.  The block descriptions now show that inclusion corresponds
componentwise to subgroup inclusion and to $0<1$ in each flag.  This proves
the lattice assertion.

For the typed assertion, sufficiency in~\eqref{eq:generaltyped} follows because
every $\phi$-replacement is a $\theta$-replacement.  Conversely, symmetry
allows carrier replacements to occupy every coordinate, so $\theta$ is an
ordinary congruence.  Suppose $x\mathrel\phi y$.  If $x,y\in G$, fixing all
other inputs at $0$ yields $x\mathrel\theta y$.  If, say,
$x=a$ and $y=g\in G$, fixing all other inputs at $a$ compares the outputs
$a$ and $g+(m-1)u=g$, hence $a\mathrel\theta g$; the case involving $b$ is
the same.  If $a\mathrel\phi b$, the all-$a$, all-$b$, and zero contexts give
respectively $a\mathrel\theta v$, $u\mathrel\theta b$, and
$u\mathrel\theta v$, whence $a\mathrel\theta b$.  Thus every pair of $\phi$
belongs to $\theta$.
\end{proof}

\begin{theorem}[Quotient-criticality criterion]
\label{thm:generalcritical}
Put $d=u-v$.  Then $F_m(G;u,v)$ is binary reducible if and only if $d=0$.
More generally,
\begin{equation}
F_m(G;u,v)/\theta_N^{\varepsilon,\delta}
\text{ is binary reducible}
\quad\Longleftrightarrow\quad
\varepsilon=1\ \text{or}\ \delta=1\ \text{or}\ d\in N.
\label{eq:quotientcriterion}
\end{equation}
Consequently $F_m(G;u,v)$ is quotient-critically irreducible if and only if
\begin{equation}
d\ne0
\quad\text{and}\quad
d\in N\text{ for every nonzero subgroup }N\leq G.
\label{eq:monolithcriterion}
\end{equation}
If $G$ is finite, condition~\eqref{eq:monolithcriterion} holds exactly when
$G$ is a cyclic group of order $p^k$ for some prime $p$ and $k\geq1$, and
$d$ has order $p$.
\end{theorem}

\begin{proof}
The reducibility criterion for a strong $m$-ary semilattice asks for compatible
choices of neutral elements in its group components
\cite[Proposition~5.3]{DevilletMathonet2021}.  The singleton branches force
the bottom choice to be both $u$ and $v$.  Hence the source is reducible
exactly when $u=v$.

The quotient by $\theta_N^{0,0}$ has the two-branch form
$F_m(G/N;u+N,v+N)$, so it is reducible exactly when $d\in N$.  If one flag is
one, its corresponding branch has been absorbed into the group quotient and
at most one singleton branch remains; choosing the image of that branch as
the neutral element gives a binary reduction.  If both flags are one, the
quotient is the $m$-fold operation of $G/N$ and is reducible.  This proves
\eqref{eq:quotientcriterion}.  The identity congruence is
$\theta_{\{0\}}^{0,0}$, so~\eqref{eq:monolithcriterion} follows.

Assume now that $G$ is finite and~\eqref{eq:monolithcriterion} holds.  A
subgroup of prime order contained in $\langle d\rangle$ must contain $d$;
therefore $d$ itself has prime order, say $p$.  A subgroup of order $q$ for
any prime divisor $q$ of $|G|$ must also contain $d$, forcing $q=p$; hence
$G$ is a $p$-group.  A noncyclic finite abelian $p$-group has at least two
distinct subgroups of order $p$, one of which does not contain $d$.  Thus $G$
is cyclic.  Conversely, every nonzero subgroup of a cyclic group of order
$p^k$ contains its unique subgroup of order $p$, so it contains every element
of order $p$, including $d$.
\end{proof}

\section{A four-point family in every odd arity}\label{sec:family}

Fix an odd integer $m\geq3$ and let
\[
X=\{1,2,3,4\},\qquad
\eps(1)=\eps(4)=0,\qquad \eps(2)=\eps(3)=1.
\]
Put $\eta(0)=4$ and $\eta(1)=3$.  Define $H_m:X^m\to X$ by
\begin{equation}
H_m(x_1,\ldots,x_m)=
\begin{cases}
1,&x_1=\cdots=x_m=1,\\
2,&x_1=\cdots=x_m=2,\\
\eta\bigl(\eps(x_1)+\cdots+\eps(x_m)\bigr),&\text{otherwise},
\end{cases}
\label{eq:Hm}
\end{equation}
where the sum is taken in $\mathbb Z_2$.  For $m=3$ this is the irreducible
operation in~\cite[Example~1.1(a)]{DevilletMathonet2021}; formula
\eqref{eq:Hm} realizes the same strong-semilattice datum in every odd arity.
Under the identification $a=1$, $b=2$, $0=4$, and $1=3$ in the bottom copy
of $\mathbb Z_2$, it is precisely
\[
H_m=F_m(\mathbb Z_2;0,1).
\]

\begin{lemma}[Structure and parity]\label{lem:Hstructure}
For every odd $m\geq3$, $H_m$ is a symmetric idempotent $m$-ary semigroup and
\begin{equation}
\eps(H_m(x_1,\ldots,x_m))
=\eps(x_1)+\cdots+\eps(x_m).
\label{eq:parity}
\end{equation}
It is the strong $m$-ary semilattice over the three-element meet-semilattice
$c<a,c<b$, with components
\[
X_a=\{1\},\qquad X_b=\{2\},\qquad X_c=\{3,4\}\cong\mathbb Z_2,
\]
and connecting homomorphisms $1\mapsto4$ and $2\mapsto3$ into $X_c$.
\end{lemma}

\begin{proof}
Equation~\eqref{eq:parity} is immediate in the third case of
\eqref{eq:Hm}; it also holds on the two exceptional diagonals because $m$ is
odd.  Identify $4$ with $0$ and $3$ with $1$ in $\mathbb Z_2$.  On $X_c$ the
operation is the $m$-fold group sum.  The two connecting maps are homomorphisms
because $m\cdot0=0$ and $m\cdot1=1$ in $\mathbb Z_2$.  The strong-semilattice
reconstruction gives~\eqref{eq:Hm}: a mixed tuple is transported to $X_c$ and
evaluated by parity.  This is also the specialization of
Theorem~\ref{thm:generalcongruences} to $G=\mathbb Z_2$, $u=0$, and $v=1$.
\end{proof}

For compactness, a partition is written by concatenating elements in a block
and separating blocks by vertical bars.  Define
\begin{align*}
\Delta&=1|2|3|4,&
\lambda&=1|23|4,&
\mu&=14|2|3,&
\nu&=1|2|34,\\
\alpha&=134|2,&
\pi&=14|23,&
\beta&=1|234,&
\nabla&=1234.
\end{align*}
For every $\theta\ne\Delta$ in this list, order the quotient classes as in
Table~\ref{tab:quotientreductions}.  The displayed matrix is the Cayley table
of a commutative associative operation $g_\theta$; row and column labels are
$0,1$ or $0,1,2$ in the stated class order.

\begin{table}[ht]
\centering
\small
\caption{Binary reductions of the proper quotients of $H_m$}
\label{tab:quotientreductions}
\begin{tabular}{c@{\quad}c@{\quad}c}
\hline
$\theta$ & ordered quotient classes & Cayley matrix of $g_\theta$\\
\hline
$\nabla$ & $1234$ & $\begin{pmatrix}0\end{pmatrix}$\\[2mm]
$\alpha$ & $134,2$ & $\begin{pmatrix}0&0\\0&1\end{pmatrix}$\\[3mm]
$\pi$ & $14,23$ & $\begin{pmatrix}0&1\\1&0\end{pmatrix}$\\[3mm]
$\beta$ & $1,234$ & $\begin{pmatrix}0&1\\1&1\end{pmatrix}$\\[3mm]
$\lambda$ & $1,23,4$ &
$\begin{pmatrix}0&1&2\\1&2&1\\2&1&2\end{pmatrix}$\\[4mm]
$\mu$ & $14,2,3$ &
$\begin{pmatrix}2&0&0\\0&1&2\\0&2&2\end{pmatrix}$\\[4mm]
$\nu$ & $1,2,34$ &
$\begin{pmatrix}0&2&2\\2&1&2\\2&2&2\end{pmatrix}$\\[3mm]
\hline
\end{tabular}
\end{table}

\begin{lemma}[Uniform quotient reductions]\label{lem:quotientreductions}
Let $q_\theta:X\to X/\theta$ be the class map.  For every odd $m\geq3$ and
each $\theta\in\{\nabla,\alpha,\pi,\beta,\lambda,\mu,\nu\}$,
\begin{equation}
q_\theta(H_m(x_1,\ldots,x_m))
=g_\theta^{(m)}(q_\theta(x_1),\ldots,q_\theta(x_m)),
\label{eq:quotientreduction}
\end{equation}
where $g_\theta^{(m)}$ is the associative $m$-fold product.  Hence every
listed $\theta$ is a congruence and every corresponding quotient is binary
reducible.
\end{lemma}

\begin{proof}
For $\alpha$, the quotient output is the class $2$ exactly when all inputs
belong to that class; this is the two-element meet table.  The dual statement
for $\beta$ gives the join table.  For $\pi$, equation~\eqref{eq:parity} gives
addition in $\mathbb Z_2$.

For $\lambda$, $0$ is an identity and $\{1,2\}$ is a copy of $C_2$ with
identity $2$.  Thus an $m$-fold product is $0$ only on the all-$0$ tuple and,
otherwise, is $1$ or $2$ according as the number of entries $1$ is odd or
even.  These are exactly the exceptional all-$1$ case and the two parity
cases of~\eqref{eq:Hm} modulo $\lambda$.

For $\mu$, $1$ is an identity and $\{0,2\}$ is a copy of $C_2$ with identity
$2$.  Away from the all-$1$ quotient tuple, the product is $0$ or $2$
according as the number of entries $0$ is odd or even.  Because $m$ is odd,
this is equivalent to the required parity rule, while the all-$1$ tuple
records the exceptional all-$2$ input of~\eqref{eq:Hm}.  For $\nu$, the
elements $0$ and $1$ are idempotent and every product involving two different
elements is $2$; hence its $m$-fold product is $0$ on the all-$0$ tuple, $1$
on the all-$1$ tuple, and $2$ otherwise.  These descriptions also prove
associativity of the three tables.  The one-element case is immediate, and
\eqref{eq:quotientreduction} proves both claims.
\end{proof}

\subsection{The complete congruence and automorphism classification}

\begin{corollary}[Congruence lattice]\label{thm:conlattice}
For every odd $m\geq3$,
\begin{equation}
\Con(H_m)=\{\Delta,\lambda,\mu,\nu,
\alpha,\pi,\beta,\nabla\}.
\label{eq:conlist}
\end{equation}
With equivalences ordered by inclusion, this lattice is isomorphic to the
Boolean lattice $B_3$.  Its atoms are $\lambda,\mu,\nu$, and
\[
\lambda\vee\mu=\pi,
\qquad \mu\vee\nu=\alpha,
\qquad \lambda\vee\nu=\beta.
\]
\end{corollary}

\begin{proof}
The subgroup lattice of $\mathbb Z_2$ is $B_1$.  Under the four-point
identification above, the four choices attached to the zero subgroup are
$\Delta,\mu,\lambda,\pi$, and those attached to the whole group are
$\nu,\alpha,\beta,\nabla$.  Theorem~\ref{thm:generalcongruences} therefore
gives
\[
\Con(H_m)\cong B_1\times B_2\cong B_3
\]
and gives the displayed joins.
\end{proof}

\begin{corollary}[Automorphisms]\label{cor:aut}
For every odd $m\geq3$,
\[
\Aut(H_m)=\{\operatorname{id},(1\ 2)(3\ 4)\}\cong C_2.
\]
The nontrivial automorphism interchanges $\alpha$ with $\beta$ and
$\lambda$ with $\mu$, and fixes $\pi$ and $\nu$.
\end{corollary}

\begin{proof}
The strong-semilattice decomposition in Lemma~\ref{lem:Hstructure} has a
unique nonsingleton group component $\{3,4\}$, so every automorphism preserves
that component and permutes the singleton components $\{1\},\{2\}$.  If it
fixes the singleton components, compatibility with the two connecting maps
fixes $4$ and $3$; it is the identity.  If it interchanges them, compatibility
forces $3\leftrightarrow4$, giving $(1\ 2)(3\ 4)$.  Formula~\eqref{eq:Hm}
shows that this permutation is indeed an automorphism.  Its action on
\eqref{eq:conlist} is immediate.
\end{proof}

\subsection{All typed congruence pairs}

Regard $H_m$ as an alternating two-sorted operation with carrier entries in
the odd positions and index entries in the even positions; both sorts have
underlying set $X$.  A typed congruence is a pair $(\theta,\phi)$ satisfying
the alternating replacement condition used in
Theorem~\ref{thm:generalcongruences}.

\begin{corollary}[The $36$-pair classification]\label{thm:36pairs}
For every odd $m\geq3$, a pair $(\theta,\phi)$ is a typed congruence of the
alternating core $H_m$ if and only if
\begin{equation}
\theta\in\Con(H_m)
\qquad\text{and}\qquad
\phi\subseteq\theta.
\label{eq:typedclassification}
\end{equation}
Consequently the core has exactly $36$ typed congruence pairs and exactly
eight diagonal congruences.
\end{corollary}

\begin{proof}
The classification is the specialization of
Theorem~\ref{thm:generalcongruences}.  For the count, the numbers of
equivalences refining
$\nabla,\alpha,\pi,\beta,\lambda,\mu,\nu,\Delta$ are respectively
\[
15,5,4,5,2,2,2,1.
\]
Their sum is $36$.  Diagonal pairs are $(\theta,\theta)$, one for each of the
eight carrier congruences.
\end{proof}

\subsection{Quotient-critical irreducibility}

\begin{theorem}[Quotient spectrum]\label{thm:quotientspectrum}
For every odd $m\geq3$, $H_m$ is quotient-critically irreducible.  Its seven
proper quotients form five isomorphism types:
\[
H_m/\nabla,\quad H_m/\alpha\cong H_m/\beta,
\quad H_m/\pi,\quad H_m/\lambda\cong H_m/\mu,
\quad H_m/\nu.
\]
For $m=5$, the numbers of associative binary reductions of these seven
labelled quotients are
\begin{center}
\begin{tabular}{c|ccccccc}
$\theta$&$\nabla$&$\alpha$&$\pi$&$\beta$&$\lambda$&$\mu$&$\nu$\\ \hline
number&$1$&$1$&$2$&$1$&$1$&$1$&$1$.
\end{tabular}
\end{center}
The two reductions of the parity quotient are $x+y$ and $x+y+1$ on
$\mathbb Z_2$; the other reductions are the tables in
Table~\ref{tab:quotientreductions}.
\end{theorem}

\begin{proof}
Here $H_m=F_m(\mathbb Z_2;0,1)$, and the difference of the two attachment
values is the nonzero element of the cyclic group of order two.
Theorem~\ref{thm:generalcritical} therefore proves quotient-critical
irreducibility.  The reductions in Lemma~\ref{lem:quotientreductions} identify
an explicit reduction of every proper quotient.

Corollary~\ref{cor:aut} supplies the two displayed quotient isomorphisms.
The two two-element types are not isomorphic: the output-fibre sizes of
$H_m/\alpha$ are $1$ and $2^m-1$, whereas both output fibres of $H_m/\pi$
have size $2^{m-1}$.  For the three-element quotients, the output-fibre sizes
of $H_m/\lambda$ are
\[
1,\quad (3^m-1)/2,\quad (3^m-1)/2,
\]
whereas those of $H_m/\nu$ are $1,1,3^m-2$.  Hence those two types are not
isomorphic.  This proves that the five displayed types are distinct.

For the exact arity-five counts, enumerate all binary tables on the quotient
set, retain the associative ones, and test their fivefold products against
$H_5/\theta$.  There are $1$, $8$, and $113$ associative labelled operations
on sets of orders $1$, $2$, and $3$, respectively.  Substitution gives the
counts displayed above.  The supplementary verification performs this finite
exhaustion and verifies the counts; the displayed reductions identify the
surviving tables, including both reductions of the parity quotient.
\end{proof}

\section{The complete ternary \texorpdfstring{$\Gamma$}{Gamma}-semiring lift}
\label{sec:lift}

Specialize to $m=5$ and regard $(X,X,H_5)$ as a diagonal alternating quinary
core.  Its powerset complex is
\[
\T_5=(\Pow(X),\cup,\varnothing;
      \Pow(X),\cup,\varnothing;H_5^{\#}),
\]
where
\[
H_5^{\#}[A,B,C,D,E]
=\{H_5(a,b,c,d,e):a\in A,b\in B,c\in C,d\in D,e\in E\}.
\]

\begin{corollary}[Atomic quotient spectrum]\label{thm:liftedspectrum}
The algebra $\T_5$ is an atom-total, completely additive complete atomic
Boolean ternary $\Gamma$-semiring.  It has exactly $36$ strong
atom-saturated congruence pairs.  Its eight diagonal strong atom-saturated
congruences form $B_3$.  The algebra has no atomic binary collapse, but every
quotient by a proper diagonal strong atom-saturated congruence has one.  Up to
isomorphism, these proper diagonal quotients have the five types listed in
Theorem~\ref{thm:quotientspectrum}.
\end{corollary}

\begin{proof}
Lemma~\ref{lem:Hstructure} and Corollary~\ref{cor:atomiccoord} give complete
additivity, the Boolean reducts, and atom-totality.  Proposition
\ref{prop:firstiso} transfers the typed congruence classification of
Corollary~\ref{thm:36pairs} to atom-saturated complete quotients.  Lemma
\ref{lem:collapse} translates the irreducibility and quotient-reducibility
statements of Theorem~\ref{thm:quotientspectrum} into the assertions about
atomic binary collapses.
\end{proof}

This corollary identifies the point at which reducibility changes: it is absent
in the original complete algebra but present in every proper diagonal strong
atom-saturated quotient.  In particular, these reductions arise in the
quotients rather than by descent of a binary operation on the source.

\subsection{Polynomial detection of the parity quotient}

Put
\[
E=\{1,4\},\qquad O=\{2,3\},\qquad e=\{4\},
\]
and define a binary relation on $\Pow(X)$ by
\begin{equation}
\mathcal R=\{(A,B):H_5^{\#}[A,e,B,e,e]\subseteq e\}.
\label{eq:polyrelation}
\end{equation}

\begin{proposition}[One-inequation recovery]\label{thm:polyrecovery}
The relation~\eqref{eq:polyrelation} is
\[
\begin{split}
\mathcal R={}&(\{\varnothing\}\times\Pow(X))
\cup(\Pow(X)\times\{\varnothing\})\\
&\cup\{(A,B):\varnothing\ne A,B\subseteq E\}
\cup\{(A,B):\varnothing\ne A,B\subseteq O\}.
\end{split}
\]
On singleton atoms it is exactly the parity congruence:
\[
(\{x\},\{y\})\in\mathcal R
\quad\Longleftrightarrow\quad x\mathrel\pi y.
\]
The map $\eps:X\to\mathbb Z_2$ induces the corresponding strong quotient,
whose five-ary operation is addition in $\mathbb Z_2$.  The relation
$\mathcal R$ is nonrectangular and cannot be defined by a conjunction of
conditions involving only one of the two variables.
\end{proposition}

\begin{proof}
For $a,b\in X$, the tuple $(a,4,b,4,4)$ is not one of the two exceptional
diagonals, and
\[
H_5(a,4,b,4,4)=\eta(\eps(a)+\eps(b)).
\]
It equals $4$ exactly when $a$ and $b$ have equal parity.  The setwise
inclusion is vacuous if one input is empty.  Otherwise every cross-pair must
have equal parity, which occurs exactly when both sets are contained in $E$
or both in $O$.  Equation~\eqref{eq:parity} proves that $\eps$ is a
surjective homomorphism, so Proposition~\ref{prop:firstiso} identifies its
kernel quotient with the five-ary group operation on $\mathbb Z_2$.

Finally, $(\{1\},\{1\})$ and $(\{2\},\{2\})$ belong to $\mathcal R$ but
$(\{1\},\{2\})$ does not.  A conjunction of separate unary conditions defines
a rectangle, proving the last statement.
\end{proof}

The parity congruence is distinguished not only by its group quotient but
also by its shallow polynomial visibility.  To make that statement exact,
call a relation on $X^2$ a \emph{depth-one atomic trace} if it has the form
\begin{equation}
\{(x,y):H_5(u_1,a,u_2,b,u_3)\in C\},
\label{eq:depthone}
\end{equation}
where $a,b,u_i\in X$, the three carrier entries are a permutation of
$x,y,c$ for one fixed $c\in X$, and $C\subseteq X$.

\begin{proposition}[Depth-one separation spectrum]\label{prop:depthone}
Among equivalence relations on $X$, the only depth-one atomic traces are the
universal relation $\nabla$ and the parity relation $\pi$.  Thus $\pi$ is the
unique nontrivial proper congruence of $H_5$ recoverable by a single
depth-one inclusion of the form~\eqref{eq:depthone}.
\end{proposition}

\begin{proof}
There are six placements of $x,y,c$, $4^3$ choices of $a,b,c$, and $16$
choices of $C$.  Substitution in~\eqref{eq:Hm} reduces these $6144$
expressions to $12$ distinct traces.  Only two of the $12$ are reflexive:
the $16$-pair universal relation and the eight-pair relation
\[
\{11,14,41,44,22,23,32,33\},
\]
where $ij$ denotes $(i,j)$.  The latter is $\pi$.  The supplementary program
performs the complete enumeration; the reduction to the $12$ traces follows
directly by separating the two exceptional diagonals from the parity clause
of~\eqref{eq:Hm}.
\end{proof}

\section{Conclusion}

The two-branch construction produces a tractable class of symmetric polyadic
semigroups whose congruences retain both group and branch information.  For
every admissible $G,u,v$, its congruence lattice is
$\operatorname{Sub}(G)\times B_2$, and its typed congruence pairs are exactly
the refinements $(\theta,\phi)$ with $\phi\subseteq\theta$.  The
quotient-by-quotient reducibility criterion shows, in the finite case, that
quotient-critical irreducibility occurs precisely for cyclic $p$-groups with
$u-v$ of order $p$.

The four-point family is the first specialization of this criterion.  Its
ordinary congruence lattice is $B_3$, its alternating form has $36$ typed
congruence pairs, and its automorphism group is $C_2$.  Its seven proper
quotients are reducible and fall into five isomorphism types, with exact
reduction counts determined in arity five.

For arity five these facts lift to a $16$-element complete atomic ternary
$\Gamma$-semiring.  Its source algebra admits no atomic binary collapse, but
every proper diagonal strong atom-saturated quotient does.  The parity
quotient is additionally detected by one depth-one polynomial inequation and
is the only nontrivial proper congruence represented by a trace of the form
\eqref{eq:depthone}.  Thus the group-theoretic quotient criterion, the
two-sorted congruence classification, and the complete-algebra lift describe
the same transition from irreducibility to reducibility at three levels.

Three concrete problems remain: classify typed congruence pairs for other
finite semilattice diagrams of polyadic groups; determine the quotient
isomorphism and reduction spectra throughout the cyclic $p$-group families;
and compute the least polynomial depth needed to separate each of
$\lambda,\mu,\nu,\alpha,\beta$.

\section*{Data and code availability}
No empirical data were used.  The supplementary Python program verifies the
general congruence theorem on $C_3,C_4$, and $C_2\times C_2$, including all
typed pairs in the six-element ternary case.  It also verifies the five-ary
operation, enumerates all $15$ equivalences on $X$, certifies the eight
ordinary congruences and all $36$ typed congruence pairs, enumerates binary
reductions of every proper quotient, and classifies all depth-one atomic
traces.  Saved outputs are included with the source.

\section*{Funding}
The authors reported that no specific funding was received for this work.

\section*{Author contributions}
Both authors contributed to the conception, formal development, verification,
and writing of the manuscript, approved the submitted version, and accept
responsibility for its contents.

\section*{Conflict of interest}
The authors declare that they have no conflict of interest.

\end{document}